\documentclass[10pt]{article}
\usepackage{amsmath, amsfonts, amssymb, amsthm}
\usepackage{titlesec}
\usepackage{fancyhdr}
\usepackage[a4paper]{geometry}
\usepackage[cm]{fullpage}
\usepackage[colorlinks, urlcolor=black]{hyperref}
\hypersetup{citecolor=blue, linkcolor=blue}
\usepackage{makeidx}
\usepackage[all,cmtip]{xy}

\usepackage{array}
\usepackage{mathtools}
\usepackage[mathscr]{eucal}
\usepackage{mathfixs}

\newcommand{\C}{\mathbb{C}}

\newcommand{\R}{\mathbb{R}}
\newcommand{\Z}{\mathbb{Z}}

\def\a{\mathfrak{a}}
\def\b{\mathfrak{b}}
\def\H{\mathbb{H}}
\def\Im{\mathrm{Im}}
\def\PSL{\mathbf{PSL}}
\def\Re{\mathrm{Re}}

\def\SL{\mathbf{SL}}

\titleformat{\section}[hang]
{\normalfont\filright\large}{\thesection . }{0pt}
{\upshape\bfseries}

\titleformat{\subsection}[hang]
{\itshape}{\thesubsection \ - }{0pt}
{}

\newtheorem{theorem}{Theorem}[section]

\newtheorem{lemma}[theorem]{Lemma}
\newtheorem{corollary}[theorem]{Corollary}

\newtheorem*{Theorem}{Theorem}

\theoremstyle{remark}
\newtheorem*{Remark-non}{Remark}

\newtheorem*{Example-non}{Example}
\newtheorem*{Lemma-non}{Lemma}
\newtheorem*{Proposition-non}{Proposition}
\newtheorem*{Corollary-non}{Corollary}

\theoremstyle{definition}
\newtheorem*{Definition-non}{Definition}

\providecommand{\abs}[1]{\lvert#1\rvert}
\providecommand{\norm}[1]{\lVert#1\rVert}
\newcommand*\rfrac[2]{{}^{#1}\!/_{#2}}

\title{The Hodge Laplacian operator on 1-forms on $\H$ and 1-form $E_\a^1$}
\author{Otto Romero}
\date{\today}

\begin{document}

\maketitle

\begin{abstract}
As is well known, we can average the eigenfunction $y^s$ of the hyperbolic Laplacian on the hyperbolic plane by $\Gamma$ a lattice in $\SL(2,\R)$ to obtain an automorphic form, the non-holomorphic Eisenstein series $E_\a (z,s)$. 

In this note, we choose a particular eigenfunction $y^s dx$ of the Hodge-Laplace operator for 1-forms on the hyperbolic plane. Then, we average by $\Gamma$ to define a 1-form $E_\a^1 \big( (z,v), s \big)$.

We see that $E_\a^1$ admits a Fourier expansion and calculates the corresponding coefficients. Also, we evaluate the integral $\int_{\gamma} E_\a^1$ for when $\gamma$ is a lifting of horocycles and closed geodesics in the unit tangent bundle. Finally, we will obtain an analog to the Rankin-Selberg method for $E_\a^1$.


\end{abstract}
\tableofcontents

\section{Introduction}
The theory of \textit{Eisenstein series} is essential in the \textit{spectral theory of automorphic forms}. It was started by Maass and Roelcke, developed by Selberg \cite{Selberg} and  Langlands \cite{Langlands}, \cite{Moe}

\bigskip
Let $G$ be a semisimple Lie group and $\Gamma$ a lattice in $G$. Using the Eisenstein series, Selberg and Langlands studied the spectral decomposition of $L^2(\Gamma \backslash G)$. For example, let $G=\SL(2,\R)$ and $\Gamma=\SL(2,\Z)$, the quotient spaces $\Gamma \backslash \H$ (and $\Gamma \backslash G$) are non-compact and therefore admit a \textit{continuous spectrum} whose 0-eigenfunctions are given by Eisentein series.

\bigskip 
We search to study the continuous spectrum of the Hodge-Laplace operator on 1-forms over the tangent bundle to $\Gamma \backslash \H$.

\bigskip
More explicitly, we compute the Hodge-Laplace operator for 1-forms in $\H$  (Lemma \ref{OHL}, section 6), finding that a 1-eigenform is  $y^s dx$ (Corollary \ref{eigenforma}, section 6). So we define a 1-form $E_\a^1 \big( (z,v), s \big)$ for $\Gamma$ (Definition \ref{ES1}, section 3). The convergence for $\Re(s)>1$ follows from the convergence of the classical Eisenstein series and admits a Fourier expansion; we calculate the corresponding coefficients (see Theorem \ref{PEFourier}, section 3), which is our main result. Also, we calculate the integral $\int_{\gamma} E_\a^1$ for when $\gamma$ is a lifting of horocycles (see formula (\ref{InteHor}), section 4) and closed geodesics in the unit tangent bundle (Lemma \ref{InteGeo}, section 4). Finally, we will obtain an analog to the Rankin-Selberg method for $E_\a^1$ over horocycles (Lemma \ref{RSM}, section 5).

\bigskip
At the end of our work, we found that in \cite{Falliero}, the author studies the space of square-integrable 1-forms on $\Gamma \backslash \H$, with  $G=\SL(2,\R)$ and $\Gamma$ a lattice, including its relationship with the usual decomposition of $L^2(\Gamma \backslash G)$. We think our approach is complementary to \cite{Falliero}.

\section{Preliminaries and Eisenstein series}

The upper half-plane model for the hyperbolic plane $\mathbb{H}$ is given by the surface $\H \; = \; \{ (x,y) \, ; \, y > 0 \}$ with the hyperbolic Riemannian metric:
\begin{equation}
     ds^2 \, = \, \frac{dx^2+dy^2}{y^2}.
\label{RHM}
\end{equation}

\bigskip
The group $\SL(2,\R)$ acts on $\H$ by isometries via fractional linear transformations, i.e., if $g= \left( \begin{matrix}
a & b \\
c & d \\
\end{matrix} \right) \in \SL(2,\R)$ then $\displaystyle{z \longrightarrow g(z) = \frac{az+b}{cz+d}}$. 

\bigskip
We will denote $T \H$ the tangent bundle to $\H$ and $T_z \H$ the tangent space to $\H$ at $z$. In addition, we denote $S \H$ the tangent unit bundle to $\H$ and $S_z \H$ the unit (hyperbolic) circle in $T_z \H$.
We will write $e(z) := \exp(2 \pi i z)$ for all $z \in \C$.

\bigskip
The action of $\SL( 2,\R )$ extends to the tangent bundle $T \H$ as follows:
$$ g \cdot (z, v) \, = \, \big( g(z), g'(z) v \big) \, = \, \Big( \frac{az+b}{cz+d}, \frac{v}{(cz+d)^2} \Big), \; \, z \in \H, \,  v \in T_z \H.  $$

\bigskip
The action of $\SL(2,\R)$ on $\H$ is not effective, since $(-I)(z)=z$ for all $z \in \H$, where $I$ denotes the identity matrix. The group $\PSL(2,\R)$ is the set of all \textit{orientation preserving isometries} of $\H$, which is isomorphic to $\SL(2,\R)$ modulo $\{\pm I\}$.

\bigskip
We consider $\Gamma$ a non-compact \textit{Fuschian group of the first kind}, e.g., $\Gamma(N)$ (\textit{principal congruence subgroup of level} $N$) or $\Gamma_0(N)$ (\textit{congruence subgroup of level} $N$) defined as follows:
\[
\Gamma(N) := \left \{ \left( \begin{matrix}
a & b \\
c & d \\
\end{matrix} \right) \in \SL(2,\Z) \, ; \,  
\left( \begin{matrix}
a & b \\
c & d \\
\end{matrix} \right) \equiv 
 \left( \begin{matrix}
1 & 0 \\
0 & 1 \\
\end{matrix} \right) \text{ mod } N \right \},
\]
\[
\Gamma_0(N) := \left \{ \left( \begin{matrix}
a & b \\
c & d \\
\end{matrix} \right) \in \SL(2,\Z) \, ; \,  
\left( \begin{matrix}
a & b \\
c & d \\
\end{matrix} \right) \equiv 
 \left( \begin{matrix}
\ast & \ast \\
0 & \ast \\
\end{matrix} \right) \text{ mod } N \right \}.
\]
In particular, $\Gamma(1)=\SL(2,\Z)$.

\bigskip
Let $\a$ be a cusp for $\Gamma$; the stabilizer $\Gamma_\a$ is an infinite cyclic group. If we denote the generator by $\gamma_\a$, then 
$$ \Gamma_\a := \{ \gamma \in \Gamma \, ; \, \gamma \a = \a  \} = \big \langle \gamma_\a \big \rangle.  $$

\bigskip
If $\b$ denotes another (o the same) cusp, it is said that $\a$ and $\b$ are equivalent if $\b =\gamma \a$ for some $\gamma \in \Gamma$. In this case, we have $\Gamma_\a = \Gamma_\b$. For each cusp we also consider the transformations $\sigma_\a \in \Gamma$ such that
$$ \sigma_\a (\infty) = \a,
\hspace{2cm}
\sigma_\a^{-1} \Gamma_\a \sigma_\a =
\left( \begin{matrix}
1 & \Z \\
0 &  1 \\
\end{matrix} \right).
$$

\bigskip
The quotient space $\Gamma \backslash \H$ is a \textit{non-compact hyperbolic orbifold of finite volume} with finite cusps. For the preliminaries of hyperbolic geometry, Fuschian groups of the first kind, modular orbifold, see, e.g., \cite{Linda}, \cite{Iwaniec}, and \cite{Verjovsky}.

\bigskip
Let $\a$ be a cusp for $\Gamma$, the stabilizer $\Gamma_\a$ is an infinite cyclic group. If we denote the generator by $\gamma_\a$, then 
$$ \Gamma_\a := \{ \gamma \in \Gamma \, ; \, \gamma \a = \a  \} = \big \langle \gamma_\a \big \rangle.  $$

\bigskip
Let $f: \H \times \{ s \in \C \, ; \, \Re(s) > 1 \} \longrightarrow \C$ be the map given by $f(z,s):= \Im (z)^s = y^s$, where $z=x+iy$. The function $f$ is a eigenfunction of hyperbolic Laplacian $\triangle^0$, i.e., 
\begin{equation}
 \triangle^0 f \, = \, s(1-s) f,
\label{eigen}
\end{equation}
where $ \triangle^0 \, = \, - y^2 \Big( \frac{\partial^2 }{\partial x^2} + \frac{\partial^2 }{\partial y^2} \Big)$.

\bigskip
The \textit{Eisenstein series} $E_\a (z,s)$ for $\Gamma$ associated with the cusp $\a$ is defined as follows:
\[
    E_\a (z,s) \, = \, \sum_{\gamma \in  \Gamma_\a \backslash \Gamma} f \big( \sigma_\a^{-1} \gamma z,s \big) \, = \, \sum_{\gamma \in  \Gamma_\a \backslash \Gamma} \Im \big( \sigma_\a^{-1} \gamma z \big)^s.
\]

\bigskip
The series $E_\a (z,s)$ is well-defined and $\Gamma$-invariant, i.e.,
\[
  E_\a \big(\sigma(z),s\big) \, = \, E_\a (z,s), \; \, \sigma \in \Gamma.
\]

\bigskip
The series converges absolutely (and uniformly in compact subsets) for $\Re(s) > 1$;  see for a proof, \cite{Selberg} or \cite{Kubota}. A fundamental property of $E_\a (z,s)$ is that it admits a meromorphic continuation (in the variable $s \in \C$) to the whole complex plane; moreover, they are real analytic in the variable $z \in \H$. From (\ref{eigen}) we can verify that:
$$  \triangle^0 \, E_\a (z,s) \, = \, s(1-s) \, E_\a(z,s). $$

\bigskip
The following Fourier expansion applies:
$$    E_\a \big( \sigma_\b z,s \big) \, = \, \sum_{n \in \Z}  C_{\a \b, n}(y,s) \, e(nx), $$
explicitly the coefficients are given by:
$$ C_{\a \b, 0}(y,s) \, = \, \delta_{\a \b} y^s + \frac{ \sqrt{\pi} \, \Gamma \big( s-\frac{1}{2} \big) }{ \Gamma(s) } \, \varphi_{\a \b} (s) \, y^{1-s}, $$ 
with
$$ \delta_{\a \b}  \, = \, 
\begin{cases}
1 & \text{if  $\a = \b$ } \\
0 & \text{if  $\a \neq \b$ }
\end{cases} 
$$
and for $n \neq 0$ we have
$$ C_{\a \b, n}(y,s) \, = \, \frac{ 2 \pi^s \sqrt{y} }{ \Gamma(s) } \cdot \abs{n}^{s- \rfrac{1}{2}} \cdot K_{s-\rfrac{1}{2}} \big( 2 \pi \abs{n} y \big) \cdot \varphi_{\a \b} (n,s), $$
where $K$ is the \textit{modified Bessel function of the second kind}, $\varphi_{\a \b} (s)$ and $\varphi_{\a \b}(n,s)$ are the \textit{Dirichlet series} given by:
\begin{equation}
 \varphi_{\a \b} (s) \, := \, \sum_{c>0} c^{-2s} \# \Big \{ d \in [0,c) \; ; \;  
\left( \begin{matrix}
\ast  & \ast \\
c & d
\end{matrix} \right) \in \sigma_\a^{-1} \Gamma \sigma_\b
\Big \},
\label{SD1}
\end{equation}
\begin{equation}
  \varphi_{\a \b} (n,s) \, := \, \sum_{c>0} c^{-2s}  \sum_{\substack{ d \in [0,c)  \\ \left( \begin{matrix}
\ast  & \ast \\
c & d
\end{matrix} \right) \in \sigma_\a^{-1} \Gamma \sigma_\b  }} e\big( n \tfrac{d}{c} \big).
\label{SD2}
\end{equation}

\vspace{1.5cm}
\section{1-form \texorpdfstring{$E_\a^1$}{Lg}}

We recall some geometrical facts relative to the hyperbolic metric,
$$ \left \langle v_z, w_z \right \rangle \, = \, \frac{1}{\Im(z)^2} \, v \cdot w, \; \, z \in \C, \, v_z,w_z \in T_z \H, $$
where $\cdot$ denote the usual Euclidean inner product of $\R^2$ and $\left \langle \cdot, \cdot \right \rangle$ denote the hyperbolic metric. The following equality is satisfied:
\[ 
    \norm{v_z}^{\text{H}} \, = \, \sqrt{\left \langle v_z, v_z \right \rangle}
    \, = \, \frac{ \norm{v} }{\Im(z)}, \; \, z \in \H, \, v_z \in T_z \H,
\]
where $\norm{ \cdot }$ denotes the usual norm of $\R^2$ and $\norm{ \cdot }^{\text{H}}$ denotes the hyperbolic norm. If there is no confusion, we will omit the base point $z$. 

\bigskip
We will write $(z,\theta) \in S_z \H$ by the unit hyperbolic vector $v_z(\theta)$ on $T_z \H$ where $\theta$ is measured from the vertical counter-clockwise. That is to say,
$$ v_z(\theta) = \Im(z) \cdot  (-\sin \theta , \cos \theta ). $$
For a vector $v_z \in T_z \H$ we write $ v_z = \norm{v_z}^{\text{H}} \cdot v_z(\theta)$ for $\theta \in [0,2\pi)$. Sometimes we will just write $v(\theta)$ instead of $v_z(\theta)$.

\vspace{1cm}
The action of $\SL(2,\R)$ on $S \H$ is given by:
\begin{equation}
\gamma (z,\theta) = \big( z', \theta' \big)
\; \iff \;  
\gamma \big( z, v_z(\theta) \big) = \big( z',  v_{z'}(\theta') \big), 
\label{Accion1}
\end{equation}
where
$$ z'= \frac{az+b}{cz+d},
\; \, 
\theta '= \theta - 2 \, \text{arg} (cz+d), 
\; \, \text{and  } 
\gamma = \left( \begin{matrix}
a  & b \\
c  & d
\end{matrix} \right) \in \SL(2,\R). $$

\bigskip
From now on, we will denote by $\alpha^1$ the 1-form defined by the formula:
$$ \alpha^1 (z,v_z) \, := \, \Im (z)^s \cdot dx(v_z), \; \, z \in \H, \, v_z \in T_z \H, \, s \in \C. $$

\begin{Remark-non}
The 1-form $\alpha^1$ depends on $s$, but we will omit in writing the dependence not to burden the notation.
\end{Remark-non}

\bigskip
Let's see how to evaluate explicitely  the 1-form $\alpha^1$, by (\ref{Accion1}) 
\begin{equation}
 \alpha^1 \big( \gamma(z,\theta) \big) 
 \, = \, -\frac{y^{s+1}}{ \abs{ cz+d }^{2s+2} } \cdot  \sin \big( \theta - 2 \, \text{arg} (cz+d) \big), \; \, \gamma = \left( \begin{matrix}
a  & b \\
c  & d
\end{matrix} \right) \in \SL(2,\R).     
\label{Eva0}
\end{equation}
For a vector $v_z \in T_z \H$ we have:
\[ 
 \alpha^1 \big( \gamma (z,v_z) \big) 
 \, = \, \norm{v_z}^{\text{H}} \cdot  \alpha^1 \big( \gamma(z,\theta) \big).
\]

\bigskip
\begin{Definition-non}
Let $\Omega := \left \{ \big( z, v_z \big) \in T_z \H \, ; \,  \norm{v_z}^{\text{H}} \leq 1 \right \}$, i.e., the unit disk bundle associated to the tangent bundle of the hyperbolic plane. The 1-form $E_\a^1 \big( (z, v_z),s \big)$  for $\Gamma$ associated to cusp $\a$ with $(z, v_z) \in \Omega $, $s \in \C$ such that $\Re(s)>1 $, is defined by:
\begin{equation}
      E_\a^1 \big( (z, v_z),s \big) \, = \, \sum_{\gamma \in \Gamma_\a \backslash \Gamma} \alpha^1 \big( \sigma_\a^{-1} \gamma ( z, v_z ) \big).
\label{ES1}
\end{equation}
\end{Definition-non}

\bigskip
Let $\gamma  \in \Gamma_\a \backslash \Gamma$, then there exists $n \in \Z$ and $g \in \Gamma$ such that $\gamma = \gamma_\a^n g$. We have $\sigma_\a^{-1} \gamma = \big( \sigma_\a^{-1} \gamma_\a^n \sigma_\a \big) \sigma_\a^{-1} g$, but $\sigma_\a^{-1} \gamma_\a^n  \sigma_\a (\infty) = \infty$. Then $\sigma_\a^{-1} \gamma_\a^n \sigma_\a$ is a translation of the form $z \to z+m$, with $m \in \Z$. Therefore
\begin{equation}
\alpha^1 \big( \sigma_\a^{-1} \gamma ( z, v_z ) \big) =
   \alpha^1 \big( \sigma_\a^{-1} g ( z, v_z ) \big), \; \, (z,v_z) \in T_z \H.
\label{BienDef1}
\end{equation}
The equation in (\ref{BienDef1}) implies that the sum in (\ref{ES1}) is well defined.

\vspace{1cm}
The 1-form $E_\a^1 \big( (z, v_z),s \big)$ is $\Gamma$-invariant, i.e.,
\[
  E_\a^1 \big( \lambda(z, v_z),s \big) \, = \, E_\a^1 \big( (z, v_z),s \big), \; \, \lambda \in \Gamma.
\]

\bigskip
Now we check the convergence, for $(z, v_z) \in \Omega$, $\Re(s)>1 $, by (\ref{ES1}) and (\ref{Eva0})
\begin{equation}
\left| E_\a^1 \big( (z, v_z),s \big) \right| \leq 
    \norm{v_z}^{\text{H}} \sum_{\gamma \in \Gamma_\a \backslash \Gamma} 
    \left| \Im \big( \sigma_\a^{-1} \gamma z \big)^s \right|
 \leq
 \sum_{\gamma \in \Gamma_\a \backslash \Gamma} \Im \big( \sigma_\a^{-1} \gamma z \big)^{\Re(s)}.
\label{Conv}
\end{equation}
Therefore, the convergence of the 1-form $E_\a^1$ follows from the convergence of the series on the right in (\ref{Conv}), which is a consequence of the convergence of the Eisenstein series $E_\a (z,s)$ for $\Re(s)>1$, see \cite{Selberg} or \cite{Kubota} for a proof.

\bigskip
It is pertinent to comment now that in this paper, we will not deal with the problem of meromorphic continuation of $E_\a^1$.

\bigskip
Before calculating the coefficients of the Fourier expansion of $E_\a^1$, we review the \textit{Kloosterman sums}, which appear, for example, in the Fourier coefficients of automorphic series.  The Kloosterman sums for $\Gamma$ are defined as
\begin{equation}
\displaystyle{ \mathcal{S}_{\a\b} (n,m;c) \, := \, \sum_{ \left( \begin{matrix}
a  & \ast \\
c  & d
\end{matrix} \right) \in B  \backslash \sigma_\a^{-1} \Gamma \sigma_b / B } e \Big( m \, \frac{a}{c} + n \, \frac{c}{d} \Big)},
\label{Kloos}
\end{equation}
with $n,m \in \Z$ and $B:= \left \{  \left( \begin{matrix}
1  & k \\
0  & 1
\end{matrix} \right) \; ; \; k \in \Z \right \}$. We rewrite the series in (\ref{SD1}) and (\ref{SD2})
\begin{equation}
\varphi_{\a \b}(s) \, = \, \sum_{c >0} \, c^{-2s-2} \mathcal{S}_{\a\b} (0,0;c) \hspace{2cm} \varphi_{\a \b}(n,s) \, = \, \sum_{c >0} \, c^{-2s-2} \mathcal{S}_{\a\b} (n,0;c),
\label{FL1}
\end{equation}

\begin{equation}
\varphi_{\a \b,1}(n,s) \, := \, \sum_{c >0} \, c^{-2s-2} \sum_{ \left( \begin{matrix}
\ast  & \ast \\
c  & d
\end{matrix} \right) \in B  \backslash \sigma_\a^{-1} \Gamma \sigma_b / B } \cos \Big(2\pi n \, \frac{c}{d} \Big), 
\label{FL2A}
\end{equation}
\begin{equation}
\varphi_{\a \b,2}(n,s) \, := \, \sum_{c >0} \, c^{-2s-2} \sum_{ \left( \begin{matrix}
\ast  & \ast \\
c  & d
\end{matrix} \right) \in B \backslash \sigma_\a^{-1} \Gamma \sigma_b / B } \sin \Big(2\pi n \, \frac{c}{d} \Big),
\label{FL2B}
\end{equation}
i.e.,
$$ \varphi_{\a \b}(n,s) \, = \, \varphi_{\a \b,1}(n,s) + i \, \varphi_{\a \b,2}(n,s), $$
and where $ B \backslash \sigma_\a^{-1} \Gamma \sigma_b / B$ denotes the  \textit{double coset decomposition}, see e.g. \cite[Chapter 2.5]{Iwaniec}.

\bigskip
The following result determines the Fourier coefficients of $E_\a^1$ (compare \cite[Theorem 3.4]{Iwaniec} or \cite[Chapter 2.2]{Kubota}):
\begin{theorem}
Let $z=x+iy \in \C$ and $v_z= v_z(\theta) \in T_z \H$ for some $\theta \in [0,2\pi)$ a unit hyperbolic vector. In addition, we consider $\a$, $\b$ cusps of $\Gamma$. Then for $\Re(s) > 1$ we have:
$$  E_\a^1 \Big( \sigma_\b \big(z, v_z(\theta) \big),s \Big) 
\, = \,  
 \delta_{\a \b} \, y^{s+1} \sin \theta  
\, + \, \sqrt{\pi} \, y^{-s} \sin \theta \, \frac{1}{s+1} \frac{ \Gamma ( s+\rfrac{1}{2} )}{\Gamma(s)} \cdot \varphi_{\a \b}(s) \hspace{6.5cm} $$
$$
- \, 2 \pi^{s+1} \, y^{\rfrac{1}{2}} \sin \theta \cdot \Gamma(s+1)^{-1} \, \sum_{n \neq 0} e(nx) \cdot \abs{n}^{s+\rfrac{1}{2}} \cdot K_{s + \rfrac{1}{2}} \big( 2\pi \abs{n} y \big) \cdot \varphi_{\a \b}(n,s) $$
$$ + \, 4 \pi^{s+2} \, y^{\rfrac{3}{2}} \sin \theta \cdot \Gamma(s+2)^{-1} \, \sum_{n \neq 0} e(nx) \cdot  \abs{n}^{s+\rfrac{3}{2}}  
\cdot K_{s+\rfrac{3}{2}} \big( 2\pi \abs{n} y \big) \cdot \varphi_{\a \b}(n,s)  $$
$$ \hspace{2cm} + \, 8 \pi^{s+1} \,  y^{\rfrac{3}{2}} \cos \theta \cdot \sin(\pi s) \cdot \Gamma (-s-1) \, \sum_{n >0} \sin(2 \pi nx) \cdot n^{s+\rfrac{3}{2}} \cdot K_{-s-\rfrac{1}{2}} \big( 2\pi n y \big) \cdot \varphi_{\a \b,1}(n,s) $$
\begin{equation}
\hspace{2.15cm} + \, 8 \pi^{s+1} \,  y^{\rfrac{3}{2}} \cos \theta \cdot \sin(\pi s) \cdot \Gamma (-s-1) \, \sum_{n >0} \cos(2 \pi nx) \cdot n^{s+\rfrac{3}{2}} \cdot K_{-s-\rfrac{1}{2}} \big( 2\pi n y \big) \cdot \varphi_{\a \b,2}(n,s). 
\label{PP}
\end{equation}
\label{PEFourier}
\end{theorem}

\begin{proof}
Let 
$ \Omega_\infty := \left\{ \left( \begin{matrix}
\ast  & \ast \\
0  & \ast
\end{matrix} \right)  \in \sigma_\a^{-1} \Gamma \sigma_b  \right \}$, if $\a$ and $\b$ are not equivalent then $ \Omega_\infty = \emptyset $. If $\a$ is equivalent to $\b$ then $\Omega_\infty = B w_\infty B = w_\infty B = B w_\infty $
for some
$ w_\infty =  \left( \begin{matrix}
1  & \ast \\
0  & 1
\end{matrix} \right) \in \sigma_a^{-1} \Gamma \sigma_b$. We also consider $ \Omega_{\rfrac{d}{c}} := B \, w_{\rfrac{d}{c}} B $ for some $ w_{\rfrac{d}{c}} :=  \left( \begin{matrix}
\ast  & \ast \\
c  & d
\end{matrix} \right) \in \sigma_a^{-1} \Gamma \sigma_b$ with $c>0$.

\bigskip
By \cite[Theorem 2.7]{Iwaniec}, we have a disjoint union
\begin{equation}
 \sigma_\a^{-1} \Gamma \sigma_\b \, = \, \delta_{\a \b}  \Omega_\infty \cup \bigcup_{c >0} \bigcup_{d( \text{mod} \, c)} \Omega_{\rfrac{d}{c}} .
\label{Exp4}
\end{equation}

\bigskip
On the other hand, for (\ref{ES1})
 $$ E_\a^1 \big( \sigma_\b (z, v_z),s \big) \, = \, 
\sum_{\gamma \in \Gamma_\a \backslash \Gamma} \alpha^1 \Big( \sigma_\a^{-1} \gamma \big( \sigma_\b ( z, v_z ) \big) \Big) =
\sum_{\tau \in B \backslash \sigma_\a^{-1} \Gamma \sigma_\b} \alpha^1 \big( \tau ( z, v_z ) \big). $$

\bigskip
But by (\ref{Exp4}), and using the fact that $\alpha^1$ is $B$-invariant, we have
\begin{equation}
E_\a^1 \big( \sigma_\b (z, v_z),s \big) \, = \,  \delta_{\a \b} \, \alpha^1 (z,v_z) \, + \, \sum_{c >0} \sum_{d( \text{mod} \, c)} \sum_{n \in \Z}  \alpha^1 \big( w_{\rfrac{d}{c}} (z+n,v_{z+n} ) \big). 
\label{Exp4E}
\end{equation}

 \bigskip
Applying Poisson's summation formula we know that
\begin{equation}
 \sum_{n \in \Z}  \alpha^1 \big( w_{\rfrac{d}{c}} (z+n,v_{z+n} ) \big) \, = \,
 \sum_{n \in \Z} \int_{-\infty}^{+\infty}  \alpha^1 \big( w_{\rfrac{d}{c}} (z+t,v_{z+t} ) \big) \, e(-nt) dt. 
\label{Exp6}
\end{equation}

\bigskip
Let $\gamma_t := \left( \begin{matrix}
1  & t \\
0  & 1
\end{matrix} \right)$ with $t \in \R$, then 
$$  \alpha^1 \big( w_{\rfrac{d}{c}} (z+t,v_{z+t} ) \big) \, = \, 
\alpha^1 \big( w_{\rfrac{d}{c}} \gamma_t (z, \theta ) \big), $$
with $v_{z+t}=v_{z+t}(\theta)$. But since $ w_{\rfrac{d}{c}} \gamma_t = \left( \begin{matrix}
a  & at+b \\
c  & ct+d
\end{matrix} \right)$, by (\ref{Eva0}) we have the next equality:
$$
 \alpha^1 \big( w_{\rfrac{d}{c}} (z+t,v_{z+t} ) \big) 
 \, = - \, \frac{y^{s+1}}{ \abs{ cz+ct+d }^{2s+2} } \cdot \sin \big( \theta - 2 \, \text{arg} (cz+ct+d) \big)
 $$
\begin{equation}
 = - \, \frac{y^{s+1}}{ \abs{ cz+ct+d }^{2s+4} } \cdot \sin \theta \cdot (cx+ct+d)^2
 + \, \frac{y^{s+1}}{ \abs{ cz+ct+d }^{2s+4} } \cdot \sin \theta \cdot  c^2 y^2 
 + \, \frac{y^{s+1}}{ \abs{ cz+ct+d }^{2s+4} } \cdot 2  \cos \theta \cdot cy(cx+ct+d).
\label{Exp12}
\end{equation}

\bigskip
Substituting (\ref{Exp12}) in (\ref{Exp6}) we have
 $$  \sum_{n \in \Z}  \alpha^1 \big( w_{\rfrac{d}{c}} (z+n,v_{z+n} ) \big)  = - \, y^{s+1} \,  \sin \theta \, \sum_{n \in \Z} \int_{-\infty}^{+\infty} \frac{ (cx+ct+d)^2 }{  \abs{ cz+ct+d }^{2s+4}  }  \, e(-nt) \, dt $$
\begin{equation}
\, + \, y^{s+3} \, c^2 \, \sin \theta \, \sum_{n \in \Z} \int_{-\infty}^{+\infty} 
\frac{1}{  \abs{ cz+ct+d }^{2s+4}  }  \, e(-nt) \, dt 
\, + \,
2y^{s+2} \, c \, \cos \theta \,  \sum_{n \in \Z} \int_{-\infty}^{+\infty} 
\frac{ cx+ct+d }{  \abs{ cz+ct+d }^{2s+4}  }  \, e(-nt) \, dt. 
\label{Exp14}
\end{equation}

\bigskip
Now we simplify the integrals in (\ref{Exp14}). For this, we first change variable $t=r-x-\rfrac{d}{c}$. The following identities are satisfied:
$$  \int_{-\infty}^{+\infty} \frac{ (cx+ct+d)^2 }{ \abs{cz+ct+d}^{2s+4} }  \, e(-nt) dt \, = \, e\big(nx+ \rfrac{nd}{c}\big) \, c^{-2s-2} y^{-2s-1} \int_{-\infty}^{+\infty} \frac{1}{ (t^2 + 1)^{s+1} } \, e(-nyt) \, dt \hspace{4cm} $$
\begin{equation}
\hspace{4.5cm}  - \, e\big(nx+ \rfrac{nd}{c}\big) \, c^{-2s-2} \, y^{-2s-1} \int_{-\infty}^{+\infty} \frac{1}{(t^2 + 1)^{s+2}} \, e(-nyt) \, dt. 
\label{Fou7}
\end{equation}
\begin{equation}
\int_{-\infty}^{+\infty} \frac{1}{\abs{ cz+ct+d }^{2s+4}} \, e(-nt) \, dt \, = \,
e\big(nx+ \rfrac{nd}{c}\big) \, c^{-2s-4}   y^{-2s-3} \int_{-\infty}^{+\infty} \frac{1}{(t^2 + 1)^{s+2}} \, e(-nyt) \, dt. \hspace{3cm}
\label{Fou8}
\end{equation}
\begin{equation}
 \int_{-\infty}^{+\infty} \frac{ cx+ct+d }{  \abs{ cz+ct+d }^{2s+4}  }  \, e(-nt) dt \, =
 - \, 2i \, e\big(nx+ \rfrac{nd}{c}\big) \, c^{-2s-3}  y^{-2s-2} \int_{0}^{\infty} \frac{t}{(t^2 +1)^{s+2}} \, \sin (2\pi nyt) \, dt.
 \hspace{3cm}
\label{Fou9}
\end{equation}

\bigskip
Replacing (\ref{Fou7}), (\ref{Fou8}) and (\ref{Fou9}) in (\ref{Exp14})
$$ \displaystyle \sum_{n \in \Z} \alpha^1 \big( w_{\rfrac{d}{c}} (z+n,v_{z+n} ) \big) 
\, = \, - \, y^{-s} \, \sin \theta \sum_{n \in \Z} \, e\big(nx+ \rfrac{nd}{c}\big) \, c^{-2s-2} \int_{-\infty}^{+\infty} \frac{1}{(t^2 + 1)^{s+1}  }  \, e(-nyt) \, dt \hspace{3cm} $$
 $$  + \, y^{-s} \, \sin \theta \sum_{n \in \Z} \, e\big(nx+ \rfrac{nd}{c}\big) \, c^{-2s-2} \int_{-\infty}^{+\infty} \frac{1}{(t^2 + 1)^{s+2}  } \, e(-nyt) \, dt $$ 
$$ + \, y^{-s} \, \sin \theta \, \sum_{n \in \Z} e\big(nx+ \rfrac{nd}{c}\big) \, c^{-2s-2} \int_{-\infty}^{+\infty} \frac{1}{(t^2 + 1)^{s+2}  }  \, e(-nyt) \, dt $$
\begin{equation}
\hspace{0.75cm} - \, 4i y^{-s} \,  \cos \theta \, \sum_{n \in \Z}  e\big(nx+ \rfrac{nd}{c}\big) \, c^{-2s-2}  \int_{0}^{\infty} \frac{t}{(t^2 +1)^{s+2}} \, \sin (2\pi nyt) \, dt. 
\label{Exp24}
\end{equation}

\bigskip
Substituting (\ref{Exp24}) in (\ref{Exp4E}) we have
$$ E_\a^1 \big( \sigma_\b (z, v_z),s \big)  $$
$$ = \,  \delta_{\a \b} \, y^{s+1} \sin \theta  
\, - \, y^{-s} \sin \theta \, \sum_{n \in \Z} 
e(nx) \, \Bigg[ \int_{-\infty}^{+\infty} \frac{1}{ (t^2 + 1)^{s+1}  }  \, e(-nyt) \, dt \Bigg] \Bigg[ \sum_{c >0} \, c^{-2s-2} \sum_{ \left( \begin{matrix}
a  & \ast \\
c  & d
\end{matrix} \right) \in B \backslash \sigma_\a^{-1} \Gamma \sigma_b / B } \, e\Big( \frac{nd}{c}\Big) \Bigg]  $$
$$ + \, 2 y^{-s} \sin \theta \, \sum_{n \in \Z} e(nx) \, \Bigg[ \int_{-\infty}^{+\infty} \frac{1}{ (t^2 + 1)^{s+2}  }  \, e(-nyt) \, dt \Bigg] \Bigg[  \sum_{c >0} \, c^{-2s-2} \sum_{ \left( \begin{matrix}
a  & \ast \\
c  & d
\end{matrix} \right) \in B \backslash \sigma_\a^{-1} \Gamma \sigma_b / B } \, e\Big(\frac{nd}{c}\Big) \Bigg] \hspace{1.5cm} $$
\begin{equation}
- \, 4i y^{-s} \cos \theta \, \sum_{n \in \Z} e(nx) \,  \bigg[ \int_{0}^{\infty} \frac{t}{ (t^2 +1)^{s+2} } \, \sin (2\pi nyt) \, dt \Bigg] \Bigg[ \sum_{c >0} \, c^{-2s-2} \sum_{ \left( \begin{matrix}
a  & \ast \\
c  & d
\end{matrix} \right) \in B \backslash \sigma_\a^{-1} \Gamma \sigma_b / B } \,  e\Big(\frac{nd}{c}\Big) \bigg].
\hspace{0.75cm}
\label{Exp40}
\end{equation}

\bigskip
We recognize in (\ref{Exp40}) the Kloosteman sums (\ref{Kloos}) and its associated series (\ref{FL1}), (\ref{FL2A}) and (\ref{FL2B}). In addition, we have the following well-known integral:
\begin{equation}
 \int_{-\infty}^{\infty} \frac{1}{( t^2 + 1 )^{s}} \, e(-nyt) \, dt  \, = \, 
\begin{cases}
2 \pi^s  \abs{n}^{s - \rfrac{1}{2}}  \, y^{s - \rfrac{1}{2}}
\, \Gamma(s)^{-1} K_{s - \rfrac{1}{2}} \big( 2\pi \abs{n} y \big) & \text{if  $n \neq 0$ } \\
  \sqrt{\pi} \,  \frac{ \Gamma ( s - \rfrac{1}{2} )}{ \Gamma(s) } & \text{if $ n = 0$ }
\end{cases} 
\label{IntInfinita}
\end{equation}

\bigskip
Then by (\ref{IntInfinita})
$$  E_\a^1 \big( \sigma_\b (z, v_z),s \big) \, = \, \delta_{\a \b} \, y^{s+1} \sin \theta \, - \, y^{-s} \sin \theta \, \sqrt{\pi} \,  \frac{ \Gamma ( s+\rfrac{1}{2} )}{ \Gamma(s+1) } \, \varphi_{\a \b}(s) 
\, + \, 2 y^{-s} \sin \theta \, \sqrt{\pi} \, \frac{ \Gamma (s+\rfrac{3}{2} )}{ \Gamma(s+2) } \, \varphi_{\a \b}(s) $$
$$
- \,  2 \pi^{s+1} \, y^{\rfrac{1}{2}} \sin \theta \, \Gamma(s+1)^{-1} \, \sum_{n \neq 0} 
e(nx) \cdot \abs{n}^{s+\rfrac{1}{2}} \cdot K_{s + \rfrac{1}{2}} \big( 2\pi \abs{n} y \big) \cdot \varphi_{\a \b}(n,s) $$
$$
 + \, 4 \pi^{s+2} y^{\rfrac{3}{2}} \sin \theta \, \Gamma(s+2)^{-1} \, \sum_{n \neq 0} e(nx) \cdot   \abs{n}^{s+\rfrac{3}{2}} \cdot K_{s+\rfrac{3}{2}} \big( 2\pi \abs{n} y \big) \cdot \varphi_{\a \b}(n,s) $$
\begin{equation}
 \, - \, 4i  y^{-s} \cos \theta \, \sum_{n >  0} e(nx) \, I_n(s) \, \Bigg[ \sum_{c >0} \, c^{-2s-2}  \mathcal{S}_{\a\b} (n,0;c) \Bigg] 
 - \, 4i  y^{-s} \cos \theta \, \sum_{m > 0} e(-mx) \, I_{-m}(s) \, \Bigg[ \sum_{c >0} \, c^{-2s-2}  \mathcal{S}_{\a\b} (-m,0;c) \bigg]
\label{C4}
\end{equation}
where 
$$ I_n(s) \, := \, \int_{0}^{\infty} \frac{t}{ (t^2 +1)^{s+2}} \, \sin (2\pi nyt) \, dt. $$

\bigskip
We simplify the sums of the last two summands in (\ref{C4}), a sum that we denote by $S$, using that $I_{-m} (s) =-I_m(s)$ for any $m \in \Z$ we see that 
$$ S \, = \, \sum_{n >  0}  I_n(s) \bigg[ \sum_{c >0} \, c^{-2s-2} \, e(nx) \, \mathcal{S}_{\a\b} (n,0;c) \bigg] 
\, - \, \sum_{n > 0}  I_n(s) \bigg[ \sum_{c >0} \, c^{-2s-2} \, e(-nx) \, \mathcal{S}_{\a\b} (-n,0;c) \bigg] \hspace{2cm} $$
$$ \, = \, 2i \sum_{n >  0}  I_n(s) \bigg [ \cos(2 \pi nx) \,  \sum_{c >0} \, c^{-2s-2} \sum_{  \left( \begin{matrix}
a  & \ast \\
c  & d
\end{matrix} \right) \in B  \backslash \sigma_\a^{-1} \Gamma \sigma_b / B } \sin \Big( 2 \pi n \, \frac{c}{d} \Big) \hspace{4cm} $$
\begin{equation}
 + \, \sin(2 \pi nx) \,  \sum_{c >0} \, c^{-2s-2} \sum_{  \left( \begin{matrix}
a  & \ast \\
c  & d
\end{matrix} \right) \in B  \backslash \sigma_\a^{-1} \Gamma \sigma_b / B } \cos \Big( 2 \pi n \, \frac{c}{d} \Big) \bigg ].
\label{R3}
\end{equation}

\bigskip
We identify the sums on the right-hand side of the identity (\ref{R3}) as the Dirichlet series in (\ref{FL2A}) and (\ref{FL2B}), then
\begin{equation}
S \, = \, 2i \sum_{n >  0}  I_n \Big[ \cos(2 \pi nx) \cdot \varphi_{\a \b,2}(n,s) \, + \, \sin(2 \pi nx) \cdot \varphi_{\a \b,1}(n,s) \Big].
\label{Exp47}
\end{equation}

\bigskip
For $n > 0$, $\Re(s) > -\tfrac{3}{2}$, we have
\begin{equation}
 I_n(s) \, = \, \int_{0}^{\infty} \frac{t}{ (t^2 +1)^{s+2}} \, \sin (2\pi nyt) \, dt
\, = \, \pi^{s+1} y^{s+\rfrac{3}{2}} n^{s+\rfrac{3}{2}} \sin(\pi s) \cdot \Gamma (-s-1) \cdot K_{-s-\rfrac{1}{2}} \big( 2\pi n y \big).
\label{ExpF14}
\end{equation}
The integral in (\ref{ExpF14}) is obtained from the integral 5 (page 442) in \cite{Table} by elementary changes.

\bigskip
On the other hand, an elementary calculation gives us that if $\Re(s)>0$ then
\begin{equation}
 \frac{ \Gamma ( s+\rfrac{1}{2} )}{ \Gamma(s+1) } \, - \, 2 \,\frac{ \Gamma (s+\rfrac{3}{2} )}{ \Gamma(s+2) } 
= - \, \frac{1}{s+1} \, \frac{ \Gamma ( s+\rfrac{1}{2} )}{\Gamma(s)}.
\label{Exp43}
\end{equation}

\bigskip
Finally, substituting (\ref{Exp47}), (\ref{ExpF14}) and (\ref{Exp43}) in (\ref{Exp40}) we obtain the identity in (\ref{PP}).
\end{proof}

\bigskip
\begin{Example-non}
Consider $\Gamma = \SL(2,\Z)$, with $\a=\b= \infty$, so we will omit the cusps. For $\Re(s)>1$ we have

$$ E^1 \big( (z, \theta),s \big) \, = \, a_0(y,\theta, s)
 \, - \, 2 \pi^{s+1} \, y^{\rfrac{1}{2}} \sin \theta \cdot \Gamma(s+1)^{-1} \, \sum_{n \neq 0} e(nx) \cdot \abs{n}^{s+\rfrac{1}{2}} \cdot K_{s + \rfrac{1}{2}} \big( 2\pi \abs{n} y \big) \cdot \varphi(n,s) $$
$$ + \, 4 \pi^{s+2} \, y^{\rfrac{3}{2}} \sin \theta \cdot \Gamma(s+2)^{-1} \, \sum_{n \neq 0} e(nx) \cdot  \abs{n}^{s+\rfrac{3}{2}}  
\cdot K_{s+\rfrac{3}{2}} \big( 2\pi \abs{n} y \big) \cdot \varphi(n,s)  $$
$$ \hspace{2cm} + \, 8 \pi^{s+1} \,  y^{\rfrac{3}{2}} \cos \theta \cdot \sin(\pi s) \cdot \Gamma (-s-1) \, \sum_{n >0} \sin(2 \pi nx) \cdot n^{s+\rfrac{3}{2}} \cdot K_{-s-\rfrac{1}{2}} \big( 2\pi n y \big) \cdot \varphi_1(n,s) $$
$$ \hspace{2.15cm} + \, 8 \pi^{s+1} \,  y^{\rfrac{3}{2}} \cos \theta \cdot \sin(\pi s) \cdot \Gamma (-s-1) \, \sum_{n >0} \cos(2 \pi nx) \cdot n^{s+\rfrac{3}{2}} \cdot K_{-s-\rfrac{1}{2}} \big( 2\pi n y \big) \cdot \varphi_2(n,s), $$
with $ \displaystyle{ \varphi (n,s) \, = \, \frac{1}{\zeta(2s)} \sum_{d \vert n} d^{-2s+1} }$ and zero coefficient
$$ a_0(y,\theta, s) \, = \,
  y^{s+1} \sin \theta  
\, + \, \sqrt{\pi} \, y^{-s} \sin \theta \, \frac{1}{s+1} \frac{ \Gamma ( s+\rfrac{1}{2} )}{\Gamma(s)} \cdot \varphi(s), $$
where $\displaystyle{ \varphi(s) = \frac{ \sqrt{\pi} \, \Gamma ( s-\rfrac{1}{2} )}{ \Gamma(s) } \, \frac{ \zeta ( 2s-1)}{ \zeta(2s) } }$. Then  $a_0(y,\theta, s)$ admits a meromorphic continuation to all $\C$ with a pole at $s=1$ and
$$ \displaystyle { \text{Res}_{s=1} \, a_0(y,\theta, s) \, = \, \frac{ \pi \sqrt{\pi} \sin \theta \cdot \Gamma( \rfrac{3}{2})}{2 y \, \zeta(2)} = \frac{3 \sin \theta}{2y} }. $$
\end{Example-non}

\vspace{1cm}
\section{Integration over liftings of horocycles and geodesics}

In this section, we choose $\Gamma=\SL(2,\Z)$ and $\a=\b=\infty$. As is well known, the unit tangent bundle to $\Gamma \backslash \H$ (modular surface) is homeomorphic to the complement of the trefoil knot in 3-sphere $\mathbb{S}^3$; this is due to Quillen, see \cite{Milnor} (page 84) for a proof. Therefore, lifts of closed geodesics on the modular surface correspond to knots in $\mathbb{S}^3$. Then, it is natural to ask about the values of the integrals of $E_\a^1$ along those curves. 

\bigskip
But first, we consider the lifting of horocycles to the unit tangent bundle parameterized as:
$$ \mathcal{H}_{y_0} := \left \{ \big( z_0, v_{z_0}(-\rfrac{\pi}{2}) \big) \, ; \, x \in [0,1]  \right \} \subset S \H, $$
with $z_0=x+iy_0$ and $y_0 >0$. 

\bigskip
By the Fourier expansion in Theorem \ref{PEFourier} for $\Re (s)>1$ we have
$$ E_\a^1 \Big( \sigma_\a \big(x+iy_0, v(-\rfrac{\pi}{2}) \big),s \Big) \, = 
- \, y_0^{s+1} \, - \, \sqrt{\pi} \, y_0^{-s} \, \frac{1}{s+1} \frac{ \Gamma ( s+\rfrac{1}{2} )}{\Gamma(s)} \cdot \varphi(s) \hspace{6cm} $$
$$
+ \, 2 \pi^{s+1} \, y_0^{\rfrac{1}{2}} \cdot \Gamma(s+1)^{-1} \, \sum_{n \neq 0} e(nx) \cdot \abs{n}^{s+\rfrac{1}{2}} \cdot K_{s + \rfrac{1}{2}} \big( 2\pi \abs{n} y_0 \big) \cdot \varphi(n,s) $$
\begin{equation}
- \, 4 \pi^{s+2} \, y_0^{\rfrac{3}{2}} \cdot \Gamma(s+2)^{-1} \, \sum_{n \neq 0} e(nx) \cdot  \abs{n}^{s+\rfrac{3}{2}}  
\cdot K_{s+\rfrac{3}{2}} \big( 2\pi \abs{n} y_0 \big) \cdot \varphi(n,s). 
\label{Horo1}
\end{equation}

\bigskip
Integrating the equation in (\ref{Horo1}) we have
\begin{equation}
\int_{\mathcal{H}_{y_0}} E_\a^1 =  \int_0^1 E_\a^1 \Big( \sigma_\a \big(x+iy_0, v(-\rfrac{\pi}{2}) \big),s \Big) \, dx = - \, y_0^{s+1} \, - \, \sqrt{\pi} \, y_0^{-s} \, \frac{1}{s+1} \frac{ \Gamma ( s+\rfrac{1}{2} )}{\Gamma(s)} \cdot \varphi(s).
\label{InteHor}
\end{equation}

\bigskip
Now, for lifts of closed geodesics, we will prove the following:
\begin{lemma}
Let $\gamma \in \Gamma$ be a hyperbolic element, $z_0$ a base point on the geodesic axis for $\gamma$. The axis determines a closed-oriented geodesic in $\Gamma \backslash \H$, which can be lifted to an oriented path in $ S \H $, which we also denote by $\gamma$. Then if $1< \Re(s) <2$ we have that
$$  \displaystyle { \int_{\gamma} E_\a^1  \, = \, \Lambda_1(s) \, \bigg[ \sum_{n >0}  n^{s-1} \, \varphi_{2}(n,s) \bigg] \, \bigg[ y^{-s+2} \, _1F_2\left( \frac{-s+2}{2} ; \frac{-2s+1}{2}, \frac{-s+4}{2} ; \frac{y^2}{4} \right) \bigg] \bigg \vert_{y_1}^{y_2} } \hspace{3cm} $$ 
$$ \hspace{3cm} - \, \displaystyle { \Lambda_2(s) \, \bigg[ \sum_{n >0}  n^{s-1} \, \varphi_{2}(n,s) \bigg] \, \bigg[  y^{s+3} \, \, _1F_2\left( \frac{s+3}{2} ; \frac{2s+3}{2}, \frac{s+5}{2} ; \frac{y^2}{4} \right) \bigg] \bigg \vert_{y_1}^{y_2} } \hspace{4cm} $$ 
whith
$$ y_1 :=\Im \big( \sigma_{\a}^{-1} z_0 \big),  \hspace{4cm} y_2 := \Im \big( \sigma_{\a}^{-1}\gamma z_0 \big) $$
$$ \Lambda_1(s) := 2^s  \pi^{s-\rfrac{1}{2}} \cdot \frac{  \tan(\pi s) \cdot \Gamma (-s-1) }{(-s+2) \, \Gamma (-s+\frac{1}{2})},  \hspace{1.25cm}
 \Lambda_2(s) := 2^{-s-1} \pi^{s-\rfrac{1}{2}} \cdot \frac{  \tan(\pi s) \cdot \Gamma (-s-1) }{(s+3) \, \Gamma (s+\frac{3}{2})}, $$
and $_1F_2$ denotes the hypergeometric function defined for $\abs{z} < 1$ by the next formula:
$$ \displaystyle {}_{1}F_{2}(a;b,c;z) \, = \, \sum _{n=0}^{\infty }{\frac {(a)_{n}}{(b)_{n} (c)_{n}}}{\frac {z^{n}}{n!}}, $$
where $(\cdot)_n$ denotes the (rising) Pochhammer symbol.
\label{InteGeo}
\end{lemma}
\begin{proof}
By a change of variable (as in \cite[Lemma 3.1]{Burrin})
\begin{equation}
 \int_{\gamma} E_\a^1 \, = \, \int_{z_0}^{\gamma z_0} E_\a^1 
 \, = \,  \int_{\sigma_{\a}^{-1} z_0}^{\sigma_{\a}^{-1}\gamma z_0} \sigma_{\a}^{\ast} \, E_{\a}^{1}. 
\label{IntG1}
\end{equation}

\bigskip
We can parameterize the (vertical) geodesic that joins $\sigma_{\a}^{-1} z_0$ with $\sigma_{\a}^{-1}\gamma z_0$ as $\alpha(y) = iy$ where $y \in [y_1,y_2]$. The angle is constant $\theta(y) = 0$  for all $y \in [y_1,y_2]$. Then
\begin{equation}
 \int_{\sigma_\a^{-1} z_0}^{\sigma_\a^{-1}\gamma z_0} \sigma_\a^{\ast} \, E_\a^{1} \, = \, \int_{y_1}^{y_2} E_\a^1 \Big( \sigma_\a \big( iy, v(0) \big),s \Big).
\label{IntG2}
\end{equation}

\bigskip
By the Theorem \ref{PEFourier}, if $\Re(s)>1$, then
\begin{equation}
 E_\a^1 \Big( \sigma_\a \big(iy, v(0) \big),s \Big) 
\, = \,   8 \pi^{s+1} \cdot \sin(\pi s) \cdot \Gamma (-s-1) \, \sum_{n >0}  n^{s+\rfrac{3}{2}} \cdot  y^{\rfrac{3}{2}} \cdot K_{-s-\rfrac{1}{2}} \big( 2\pi n y \big) \cdot \varphi_{2}(n,s). 
\label{IntG3}
\end{equation}

\bigskip
Substituting (\ref{IntG3}) in (\ref{IntG2}) we have
\begin{equation}
\int_{\sigma_b^{-1} z_0}^{\sigma_\a^{-1}\gamma z_0} \sigma_a^{\ast} \, E_\a^{1} \, = \,
  8 \pi^{s+1} \cdot \sin(\pi s) \cdot \Gamma (-s-1) \, \sum_{n >0}  n^{s+\rfrac{3}{2}} \cdot \varphi_{2}(n,s) \, \bigg[ \int_{y_1}^{y_2} y^{\rfrac{3}{2}} \cdot K_{-s-\rfrac{1}{2}} \big( 2\pi n y \big) \, dy \bigg].
\label{IntG8}
\end{equation}

\bigskip
If $\lambda \in \R$, $\nu \in \C$ with $\Re(\lambda\pm \nu) > 0$ then
$$  \int y^{\lambda-1} \, K_\nu (y) \, dy \, = \,  \frac{2^{\nu-1} \pi  \csc(\pi \nu) }{(\lambda-\nu) \, \Gamma (1-\nu)} \; y^{\lambda-\nu} \,  _1F_2\left( \frac{\lambda-\nu}{2} ; 1- \nu, \frac{\lambda-\nu}{2} + 1 ; \frac{y^2}{4} \right) \hspace{4cm} $$
\begin{equation}
 \hspace{3cm} - \, \frac{2^{-\nu-1} \pi \csc(\pi \nu) }{(\lambda+\nu) \, \Gamma (1+\nu)} \;  y^{\lambda+\nu} \, _1F_2\left( \frac{\lambda+\nu}{2} ; 1+ \nu, \frac{\lambda+\nu}{2} + 1 ; \frac{y^2}{4} \right).
\label{IntG4}
\end{equation}
The integral in (\ref{IntG4}) is obtained from the integrals and identities in \cite[Chapter 2.1 and 2.3]{Luke}.

\bigskip
By  (\ref{IntG4}), if $-3< \Re(s) < 2$ then
$$ \int_{y_1}^{y_2} y^{\rfrac{3}{2}} \, K_{-s-\rfrac{1}{2}} \big( 2\pi n y \big) \, dy  
 =  2^{s-3} \, \pi^{-\rfrac{3}{2}} \,  \frac{ n^{-\rfrac{5}{2}} \, \sec(\pi s) }{(-s+2) \, \Gamma (-s+\frac{1}{2})} \bigg[  y_2^{-s+2} \, _1F_2\left( \frac{-s+2}{2} ; \frac{-2s+1}{2}, \frac{-s+4}{2} ; \frac{y_2^2}{4} \right)  \bigg] \bigg \vert_{y_1}^{y_2}  $$
 \begin{equation}
  \hspace{2cm} - \, 2^{-s-4} \, \pi^{-\rfrac{3}{2}} \, \frac{ n^{-\rfrac{5}{2}} \, \sec(\pi s) }{(s+3) \, \Gamma (s+\rfrac{3}{2})} \bigg[ y^{s+3} \, _1F_2\left( \frac{s+3}{2} ; \frac{2s+3}{2}, \frac{s+5}{2} ; \frac{y^2}{4} \right) \bigg]  \bigg \vert_{y_1}^{y_2}.
\label{IntG5}
\end{equation}

\bigskip
Substituting (\ref{IntG5}) in (\ref{IntG8}) we obtain the formula of the Lemma \ref{InteGeo}.
 \end{proof}

\vspace{0.75cm}
\section{Rankin-Selberg method for 1-forms}

The \textit{Rankin-Selberg method} refers to the general principle that consists of integrating the product of functions $F: \H \longrightarrow \C$ (or $F: \SL(2,\R) \longrightarrow \C$) that are $\Gamma$-invariants with Eisenstein series by obtaining the Mellin transform of the constant term of the Fourier expansion of $F$, see, e.g., \cite{Zagier} and \cite{Zagier2}.

\bigskip
We will get an analog of the method by integrating 1-forms (over lifts of horocycles) resulting from the product of functions $\Gamma$-invariants $F: \SL(2,\R) \longrightarrow \C$ by the 1-form $E_\a^1$. 

\bigskip
We also will assume that $\Gamma= \SL(2,\Z)$. We consider $f: \H \longrightarrow \C$ a \textit{modular form of weight} $k$ for $\Gamma$, i.e., 
$$ f\Big( \frac{az+b}{cz+d} \Big) \, = \, (cz+d)^k \, f(z), \; \,  \left(  \begin{matrix}
a & b \\
c & d \\
\end{matrix}  \right) \in \Gamma, \, z \in \H, $$
which is \textit{holomorphic at the cusp} $\infty$. Therefore, $f(z+1) = f(z)$ for all $z \in \H$, so it admits a Fourier expansion
\begin{equation}
\displaystyle{ f(z) \, = \, \sum_{n=0}^\infty a_n \, e(nz)  }
\label{RS1}
\end{equation}
with $a_n \in \C$. 

\bigskip
Following \cite{Gelbart}, we can associate  to $f$ a function $\phi_f : \SL(2,\R) \longrightarrow \C$ defined as follows:
\begin{equation}
\displaystyle{\phi_f (g) \, = \, (ci+d)^{-k} \, f\big( g(i) \big), \; \, g= \left(  \begin{matrix}
a & b \\
c & d \\
\end{matrix}  \right) \in \SL(2,\R). }
\label{RS3}
\end{equation}
In addition, $\phi_f$ satisfies the following properties:
\begin{equation}
\displaystyle{\phi_f (\gamma g) \, = \,\phi_f (g), \; \, \gamma \in \Gamma. }
\label{RS4}
\end{equation}
\begin{equation}
\displaystyle{\phi_f \big( g k(\theta) \big) \, = \, \exp(ik\theta) \,\phi_f (g), \; \, \theta \in [0,2\pi). }
\label{RS5}
\end{equation}

\bigskip
Let $\Phi: S \H \longrightarrow \SL(2,\R)$ be the diffeomorphism defined by:
\begin{equation}
\displaystyle{ \Phi \big( x+iy, v(\theta) \big) \, = \, n(x) a(y) k(\theta), \; \, x \in \R, \, y>0, \, \theta \in [0,2\pi),}
\label{RS6}
\end{equation}
where 
$$ n(x):= \left(  \begin{matrix}
1 & x \\
0 & 1 \\
\end{matrix}  \right),
\hspace{1.5cm}
a(y):= \left(  \begin{matrix}
\sqrt{y} & 0 \\
0 & \tfrac{1}{\sqrt{y}} \\
\end{matrix}  \right),
\hspace{1.5cm}
k(\theta):= \left(  \begin{matrix}
\cos \theta & \sin \theta \\
- \sin \theta & \cos \theta \\
\end{matrix}  \right). $$

\bigskip
Suppose that $\a=\b=\infty$ let $y_0 >0$ and the lifting of horocycles to the unit tangent bundle parameterized as:
$$ \mathcal{H}_{y_0} := \left \{ \big( z, v_{z}( \theta) \big) \, ; \, x \in [0,1]  \right \} \subset S \H, $$
where $z=x+iy_0$ and $\theta \in [0,2\pi)$ is fixed.

\bigskip
\bigskip
We have the following result:
\begin{lemma}
For $s \in \C$ with $\Re(s) >1$, the integral of 1-form $\left \vert \phi_f \circ \Phi \right ( \cdot ) \vert^2 \; E_\a^1$ along $\mathcal{H}_{y_0} \backslash \Gamma$, i.e., a lift of horocycle $\mathcal{H}_{y_0}$ to the unit tangent bundle of $\Gamma \backslash \H$, is given by:
\begin{equation}
\displaystyle{  \int_{\mathcal{H}_{y_0} \backslash \Gamma} \left \vert\phi_f \circ \Phi \right ( \cdot ) \vert^2 \; E_\a^1 \, = \, - \, y_0^{k+s+1} \sin \theta \,\sum_{n=0}^\infty \, \frac{\abs{a_n}^2}{\exp (4\pi n y_0)}. }
\label{RS}
\end{equation}
\label{RSM}
\end{lemma}
\begin{proof}
By (\ref{RS4}) and the definition in (\ref{ES1}),
$$ I:= \displaystyle{ \int_{\mathcal{H}_{y_0} \backslash \Gamma} \left \vert\phi_f \Big( \Phi \big( z,v(\theta) \big) \Big) \right \vert^2 \, E_\a^1 \big( (z, v(\theta)),s \big) \, = \, \int_{\mathcal{H}_{y_0} \backslash \Gamma}  \sum_{\gamma \in \Gamma_\infty \backslash \Gamma} \left \vert\phi_f \Big( \gamma \,  \Phi \big( z,v(\theta) \big) \Big) \right \vert^2 \, 
\alpha^1 \big( \gamma ( z, v(\theta) ) \big) }. $$

\bigskip  
By unfolding trick (see \cite{Zagier}),
\begin{equation}
I = \displaystyle{ \int_{\mathcal{H}_{y_0} \backslash \Gamma_\infty}  
 \left \vert\phi_f \Big( \Phi \big( z,v(\theta) \big) \Big) \right \vert^2 \, 
\alpha^1 \big( z, v(\theta) \big) 
\, = \,  
 \int_{0}^1 
 \left \vert\phi_f \Big( \Phi \big( x+iy_0,v(\theta) \big) \Big) \right \vert^2 \, 
\alpha^1 \big( x+iy_0, v(\theta) \big) \, dx }.
\label{RS8}
\end{equation}

\bigskip
Now we develop the two factors of the integral to the right in (\ref{RS8}). We have the following chain of identities:
\begin{align}
    \displaystyle{ \left \vert\phi_f \Big( \Phi \big( x+iy_0,v(\theta) \big) \Big) \right \vert^2 }        & \, = \, 
\displaystyle{ \left \vert\phi_f \big( n(x) a(y_0) k(\theta) \big) \right \vert^2 } 
             &&  \text {by (\ref{RS6})} \nonumber \\
              & \, = \, 
\displaystyle{ \left \vert \, \exp(ik \theta) \,\phi_f  \left(  \begin{matrix}
\sqrt{y_0} & \tfrac{x}{\sqrt{y_0}} \\
0 & \tfrac{1}{\sqrt{y_0}} \\
\end{matrix}  \right) \right \vert^2 }  
      &&   \text {by (\ref{RS5})}  \nonumber \\
              & \, = \,
\displaystyle{ \left \vert \phi_f  \left(  \begin{matrix}
\sqrt{y_0} & \tfrac{x}{\sqrt{y_0}} \\
0 & \tfrac{1}{\sqrt{y_0}} \\
\end{matrix}  \right) \right \vert^2 } 
\, = \, 
\displaystyle{ y_0^k \, \left \vert  f(x+iy_0) \right \vert^2  }    
      &&  \text {by (\ref{RS3})}  \nonumber \\
               & \, = \,
\displaystyle{ y_0^k \, \sum_{n,m=0}^\infty \, a_n \, \overline{a_m} \, \exp \big( 2\pi ix(n-m) \big) \, \exp\big( -2\pi y_0(n+m) \big)
   }  && \text {by (\ref{RS1})}   \label{RS11}
\end{align}

\bigskip
By (\ref{Eva0})
\begin{equation}
\alpha^1 \big( x+iy_0, v(\theta) \big) \, = \, - y_0^{s+1} \, \sin \theta.
\label{RS10}
\end{equation}

\bigskip
Substituting (\ref{RS11}) and (\ref{RS10}) in (\ref{RS8}) we obtain the identity in (\ref{RS}).
\end{proof}

\section{Hodge operator for 1-forms on \texorpdfstring{$\H$}{Lg}}

To study the continuous spectrum of the Hodge operator $\triangle^1$ for 1-forms, we first will calculate this operator on $\H$. For the calculations we use the global chart with coordinates $\psi=(x_1,x_2)$, with $\left \{ \frac{\partial}{\partial x_1}, \frac{\partial}{\partial x_2} \right \}$ the coordinate vector fields and $\left \{ d x_1, d x_2 \right \}$ the dual 1-form basis.

\bigskip
We rewrite the metric in (\ref{RHM}) as usual as follows:
\[
g_{kj}=\left( \begin{matrix}
\rfrac{1}{(x_2)^2}  & 0 \\
0  & \rfrac{1}{(x_2)^2}
\end{matrix} \right),
\hspace{3cm}
g^{kj}=\left( \begin{matrix}
(x_2)^2  & 0 \\
0  & (x_2)^2
\end{matrix} \right).
\]

We denote by $\abs{g}$ the determinant of $g_{kj}$. Then
\begin{equation}
    \sqrt{\abs{g}} \, = \, \frac{1}{(x_2)^2}.
\label{RDet}
\end{equation}

\bigskip
The Christoffel symbols are as follows:
\begin{equation}
\Gamma_{12}^1 = \Gamma_{21}^1 = \Gamma_{22}^2 = - \frac{1}{x_2}, 
\hspace{2cm} \Gamma_{11}^2 = \frac{1}{x_2}, 
\hspace{2cm} \Gamma_{11}^1 = \Gamma_{22}^1 = \Gamma_{21}^2 = \Gamma_{12}^2 = 0.
\label{CF}
\end{equation}

\bigskip
Let $w=\displaystyle \sum_r w_r \, dx_r \in \Omega^1(\H)$, i.e., $w$ is a 1-form differentiable in $\H$, then:
\begin{equation}
 \nabla_{\frac{\partial}{\partial x_n}} w \, = \, \sum_r \frac{\partial w_r}{\partial x_n} \, dx_r    \, + \, \sum_r w_r \nabla_{\frac{\partial}{\partial x_n}} \, dx_r, \; \, n \in \{1,2\}.
\label{HL10}
\end{equation}

\bigskip
But we also have to
\begin{equation}
\nabla_{\frac{\partial}{\partial x_n}} \, dx_r \, = \, - \sum_k \Gamma_{nk}^r \, dx_k.
\label{HL11}
\end{equation}

Substituting (\ref{HL11}) in (\ref{HL10}) 
\begin{equation}
\nabla_{\frac{\partial}{\partial x_n}} w \, = \, \sum_r \frac{\partial w_r}{\partial x_n} \, dx_r \, - \, \sum_{r,k} w_r \, \Gamma_{nk}^r \, dx_k.
\label{HL12}
\end{equation}

Deriving the formula in (\ref{HL12}),
\begin{equation}
    \nabla_{\frac{\partial}{\partial x_m}} \nabla_{\frac{\partial}{\partial x_n}} w 
  \, = \, \sum_{t} 
   \bigg(  \frac{\partial^2 w_t}{\partial x_m \partial x_n} 
   \, - \, \sum_{r} \frac{\partial w_r}{\partial x_n} \, \Gamma_{mt}^r 
   \, - \,  \sum_{r} \frac{\partial w_r}{\partial x_m} \, \Gamma_{nt}^r 
   \, - \,  \sum_{r} w_r \, \frac{\partial \Gamma_{nt}^r}{\partial x_m} 
   \, + \, \sum_{r,k} w_r \, \Gamma_{nk}^r \, \Gamma_{mt}^k  \bigg) dx_t,
\label{NMN}
\end{equation}
where $n,m \in \{1,2\}$.

\bigskip
Substituting the Christoffel symbols (\ref{CF}) into (\ref{NMN}) we obtain that
\begin{equation}
 \nabla_\frac{\partial}{\partial x_1} \nabla_\frac{\partial}{\partial x_1} w \, = \,
 \bigg(  \frac{\partial^2 w_1}{\partial (x_1)^2}
 - \frac{2}{x_2}  \frac{\partial w_2}{\partial x_1}  
 - \frac{1}{(x_2)^2} w_1  \bigg) dx_1
\, + \,
\bigg(  \frac{\partial^2 w_2}{\partial (x_1)^2}
 + \frac{2}{x_2}  \frac{\partial w_1}{\partial x_1} 
 - \frac{1}{(x_2)^2} w_2  \bigg) dx_2.   
\label{N11}
\end{equation}
\begin{equation}
\nabla_\frac{\partial}{\partial x_2} \nabla_\frac{\partial}{\partial x_2} w \, = \, 
 \bigg(  \frac{\partial^2 w_1}{\partial (x_2)^2}
 + \frac{2}{x_2}  \frac{\partial w_1}{\partial x_2}   \bigg) dx_1
\, + \,
\bigg(  \frac{\partial^2 w_2}{\partial (x_2)^2}
 + \frac{2}{x_2}  \frac{\partial w_2}{\partial x_2} \bigg) dx_2. \hspace{2.5cm} 
\label{N22}
\end{equation}

\bigskip
We have the following theorem that determines the operator $\triangle^1$, see e.g. \cite[Theorem 9.41]{Petersen}.
\begin{Theorem} (Weitzenbock) The \textit{Hodge Laplacian} $\triangle^1: \Omega^1(\H) \longrightarrow \Omega^1(\H)$ satisfies
\begin{equation}
\triangle^1 w \, = \, \nabla^\ast \nabla w \, + \, \text{Ric}_g (w^\#, \cdot),
\label{HL}     
\end{equation}
where $\nabla^\ast \nabla$ is the \textit{Bochner Laplacian} given by:
\begin{equation}
\nabla^\ast \nabla \, := \, - \sum_{k,j} \bigg[ g^{kj} \, \nabla_\frac{\partial}{\partial x_k} \nabla_\frac{\partial}{\partial x_j} \, + \, \frac{1}{\sqrt{\abs{g}}} \, \frac{\partial}{\partial x_k} \Big( \sqrt{ \abs{ g } } \, g^{kj} \Big) \cdot \nabla_\frac{\partial}{\partial x_j} \bigg],
\label{BL}     
\end{equation}
with $\text{Ric}_g$ the \textit{Ricci curvature tensor}, see e.g.
\cite[Chapter 3.1.4]{Petersen}, and for $\displaystyle  w=\sum_r w_r \, dx_r \in \Omega^1(\H)$ the vector field $w^\#$ is defined as
\begin{equation}
w^\# \, := \, \displaystyle \sum_{k,j} g^{kj} \, w_k \, \tfrac{\partial}{\partial x_j}.
\label{wA}
\end{equation}
\end{Theorem}

\bigskip
By (\ref{BL}) and since the hyperbolic metric is diagonal,
\[
  \nabla^\ast \nabla w \, = \, - \, g^{11} \, \nabla_\frac{\partial}{\partial x_1} \nabla_\frac{\partial}{\partial x_1} 
  \, - \,
  g^{22} \,\nabla_\frac{\partial}{\partial x_2} \nabla_\frac{\partial}{\partial x_2} 
  \, - \,
  \frac{1}{\sqrt{\abs{g}}} \, \frac{\partial}{\partial x_1} \Big( \sqrt{\abs{g}} \, g^{11} \Big) \cdot \nabla_\frac{\partial}{\partial x_1} 
  \, - \,
  \frac{1}{\sqrt{\abs{g}}} \, \frac{\partial}{\partial x_2} \Big( \sqrt{\abs{g}} \, g^{22} \Big) \cdot \nabla_\frac{\partial}{\partial x_2},
\]
so by (\ref{RDet}) and (\ref{CF}) it follows that:
\begin{equation}
   \nabla^\ast \nabla w \, = \, - \, (x_2)^2 \, \nabla_\frac{\partial}{\partial x_1} \nabla_\frac{\partial}{\partial x_1} w \, - \, (x_2)^2 \, \nabla_\frac{\partial}{\partial x_2} \nabla_\frac{\partial}{\partial x_2} w.
\label{NN}
\end{equation}

\bigskip
Now we compute the term $\text{Ric}_g (w^\#, \cdot )$ that appears on the right-hand side of the formula in (\ref{HL}). Since $\H$ has constant sectional curvature $k$ equal to $-1$ and it is satisfied that $\text{Ric}_g = \big( \text{dim} \, \H - 1\big)kg$ then:
\begin{equation}
\text{Ric}_g \, = \, -g.
\label{Ric}
\end{equation}

\bigskip
From (\ref{wA}) and (\ref{Ric}) we see that $\text{Ric}_g(w^\#, \cdot) = - g(w^\#, \cdot)$. Evaluating in the basis $\left \{ \frac{\partial}{\partial x_1}, \frac{\partial}{\partial x_2} \right \}$ we obtain the identity:
\begin{equation}
\text{Ric}_g(w^\#, \cdot) \, = \, - \, w_1 \, dx_1 - w_2 \, dx_2 \, = \, - \, w.
\label{Ric2}
\end{equation}

\bigskip
Substituting (\ref{NN}) and (\ref{Ric2}) in (\ref{HL}) we have that
\begin{equation}
\triangle^1 w \, = \, - \, (x_2)^2 \, \nabla_\frac{\partial}{\partial x_1} \nabla_\frac{\partial}{\partial x_1} w \, - \, (x_2)^2 \, \nabla_\frac{\partial}{\partial x_2} \nabla_\frac{\partial}{\partial x_2} w \, - \, w.
\label{HL2}
\end{equation}

\bigskip
Finally, replacing the obtained identities (\ref{N11}) and (\ref{N22}) in (\ref{HL2}) and using Cartesian coordinates $x_1=x$, $x_2=y$ we obtain the following:
\begin{lemma}
Let $w \in \Omega^1(\H)$, then
\[
 \triangle^1 (w_1 \, dx + w_2 \, dy)  = 
 \bigg( - y^2 \frac{\partial^2 w_1}{\partial x^2}
 - y^2 \frac{\partial^2 w_1}{\partial y^2}
 + 2y \frac{\partial w_2}{\partial x} - 2y \frac{\partial w_1}{\partial y}  
  \bigg) dx
\, + \, \bigg( - y^2 \frac{\partial^2 w_2}{\partial x^2}
 - y^2 \frac{\partial^2 w_2}{\partial y^2}
 - 2y \frac{\partial w_1}{\partial x} - 2y \frac{\partial w_2}{\partial y}    \bigg) dy.
\]
\label{OHL}
\end{lemma}

\begin{corollary}
If $w=y^s \, dx$, with $s \in \C$, then we have
\[
 \triangle^1 w \, = \, -s(s+1) \, w.
\]
\label{eigenforma}
\end{corollary}

\begin{Remark-non}
If $w=y^s \, dx + y^s \, dy $ then $\triangle^1 w \, = \, -s(s+1) \, w$.
\end{Remark-non}

\section*{Acknowledgements}
I am grateful to thank Christian Garay\footnote{\url{cristhian.garay@cimat.mx}} (CIMAT),  Adrián Zenteno\footnote{\url{adrian.zenteno@cimat.mx}} (CIMAT) and Alberto Verjovsky \footnote{\url{alberto@matcuer.unam.mx}} (UNAM) for their help and review on the draft version of this paper. Also, I am deeply indebted to Gregor Weingart \footnote{\url{gw@matcuer.unam.mx}} (UNAM) for discussions and lessons.

\textsc{.\\ Centro de Investigación en Matemáticas, A.C. \\ Jalisco s/n. Col. Valenciana \\
 36023 Guanajuato, Mexico.}\\
E-mail: \texttt{otto.romero@cimat.mx}

\end{document}